\documentclass[10pt]{article} 
\usepackage{amssymb}
\usepackage{amsmath,amsfonts,amsthm}
\usepackage{color}
\usepackage{comment}

\let\svthefootnote\thefootnote
\newcommand\blankfootnote[1]{%
  \let\thefootnote\relax\footnotetext{#1}%
  \let\thefootnote\svthefootnote%
}
\let\svfootnote\footnote
\renewcommand\footnote[2][?]{%
  \if\relax#1\relax%
    \blankfootnote{#2}%
  \else%
    \if?#1\svfootnote{#2}\else\svfootnote[#1]{#2}\fi%
  \fi
}

\newtheorem{theorem}{Theorem}[section]
\newtheorem{lemma}[subsection]{Lemma}

\newtheorem{corollary}[theorem]{Corollary}
\newtheorem{remark}[theorem]{Remark}
\newtheorem{definition}[theorem]{Definition}

\newtheorem{conjecture}[theorem]{Conjecture}

\usepackage{fullpage}

\begin{document}

\numberwithin{equation}{section}

\title{On the fluid ball conjecture}

\author{Fernando Coutinho$^1$, Benedito Leandro$^2$ and Hiuri F. S. Reis $^3$}
\footnote[]{$^{1,2, 3}$Universidade Federal de Goiás, IME, CEP 74690-900, Goiânia, GO, Brazil.}
\footnote[]{$^1$Universidade do Estado do Amazonas, CEST, 1085, CEP 69552-315, Tefé, AM, Brazil.}
\footnote[]{$^3$Instituto Federal de Goiás, IFG, rua Formosa, loteamento Santana,	CEP 76400-000,	Uruaçu, GO, Brazil.}
\footnote[]{Email address: \textsf{fcoutinho@uea.edu.br$^1$, bleandroneto@ufg.br$^2$ and hiuri.reis@ifg.edu.br$^3$.}}
	 \footnote[]{Fernando Soares Coutinho was partially supported by PROPG-CAPES/FAPEAM.}

\date{}

\maketitle{}

\begin{abstract}
The fluid ball conjecture states that a static perfect fluid space-time is spherically symmetric. In this paper we construct a Robinson's divergence formula for the static perfect fluid space-time. Inspired by this conjecture, a rigidity result for the spatial factor of a static perfect fluid space-time satisfying some boundary conditions is proved, provided that an equation of state holds.
\end{abstract}

\vspace{0.2cm} \noindent \emph{2020 Mathematics Subject
Classification} : 53C21,  83C05, 83D05.

\vspace{0.4cm}\noindent \emph{Keywords}: Einstein equation, Asymptotically flat, Spherically symmetric.

\ 

\section{Introduction and Main Results}

The fluid ball conjecture (or Lichnerowicz’s conjecture) states that ``{\it a static stellar model is spherically symmetric}''. In \cite{andre1964}, Avez proved that a regular, stationary, complete, with Euclidean asymptotic behaviour, perfect fluid whose trajectories coincide with time lines, must be the Schwarzschild space-time. This problem was also discussed by Yau in the 1982 list of unsolved problems in General Relativity (cf. \cite{yau1982}, see also \cite{lindblom1980}). In fact, there is a whole family of related conjectures, depending on whether the extent of the fluid region is finite or infinite, and depending on the assumptions on the asymptotics of the space-time and on the equation of state (cf. \cite{beig1991,beig1992,heinzle2003,kunzle1971,lindblom1980,massod21987} and the references therein). The conjecture is proven under physically realistic conditions, but not in full generality.

This problem was widely explored by great scientists through the years and much progress has been made. Although it is considered physically evident, the most general situation for the proof of the fluid ball conjecture is still lacking. A natural idea is to consider some reasonable equation of state for the perfect fluid to show that the Riemanniann metric of the spatial factor of a static perfect fluid space-time is conformal to a metric having nonnegative scalar curvature and zero mass. Then we can invoke the rigidity part of the positive mass theorem to prove that the original Riemannian metric is conformally flat. It is known that conformal flatness implies spherical symmetry (cf. \cite{kunzle1971,lindblom1980,massod21987,massod1987,massod2007}). Another approach is to try to deduce some divergence formula (inspired by the Robinson's black hole uniqueness theorem \cite{robinson1977}) and to combine it with a good equation of state to get spherical symmetry \cite{beig1991,beig1992}.

Inspired by \cite{massod2007} and \cite{robinson1977}, our main goal is to provide a divergence formula for the static perfect fluid equations and to give a simple proof for the fluid ball conjecture proving that $(M^3,\,g)$ is conformally flat, considering a reasonable equation of state holds. It is well-known that an appropriate linear equation of state for the perfect fluid can lead us to proof the conjecture (cf. \cite{anderson2019,heinzle2003}), under some additional hypothesis.

	Static space-times are special and important global solutions to Einstein equations in general relativity. The static perfect fluid space-time is the Einstein equation 
	\begin{eqnarray*}\label{1}
	Ric_{\hat{g}}-\frac{R_{\hat{g}}}{2}\hat{g}=T
	\end{eqnarray*}
	with perfect fluid as a matter field and static space-time  $(\widehat{M}^{n+1},\hat{g})=M^{n}\times_{f}\mathbb{R}$, such that the warped metric (cf. the warped product formulas in \cite{oneill1983}) is given by
	\begin{eqnarray*}
	\hat{g}(x,\,t)=-f^{2}(x)dt^{2}+g(x),
	\end{eqnarray*}
	where $(x,\,t)\in M\times\mathbb{R}$ and $(M^{n},g)$, is an open, connected and oriented Riemannian manifold.
	 The energy-momentum stress tensor of a perfect fluid is $T=8\pi[(\mu+\rho)U_{i}U_{j}+\rho \hat{g}]$. Here, $Ric_{\hat{g}}$ and $R_{\hat{g}}$, stand for the Ricci tensor and the scalar curvature for the metric $\hat{g}$, respectively. Moreover, $\mu$ and $\rho$ are bounded measurable functions and $U_{i}$ is a unit timelike vector field. Note that $\mu$, $\rho$ and $f$ are independent of $t$. These functions are called the {\it density}, {\it pressure} and {\it lapse function}, respectively. In what follows, we characterize a {\it static perfect fluid space-time} (see \cite{beig1991,coutinho2020, kobayashi1980,leandro2019, massod1987} for instance).

\begin{definition} 
\label{defA} A Riemannian manifold $(M^{3},g)$ is said to be the spatial factor of a static perfect fluid space-time if there exist smooth functions $\mu,\,\rho,\,f>0$ on $M$ satisfying the static perfect fluid equations:
\begin{eqnarray}\label{principaleq}
    f{Ric}={\nabla^2}f+4\pi(\mu-\rho)fg
\end{eqnarray}
and
\begin{eqnarray}\label{lapla}
\Delta f=4\pi\left(\mu+3\rho\right)f,
\end{eqnarray} 
where ${R}ic$ and ${\nabla}^{2}$ stand for the Ricci and Hessian tensors for $g$, respectively. Here, $\Delta$ is the Laplacian operator for the metric tensor $g$.
\end{definition}

The above definition implies that the scalar curvature $R$ for the metric $g$ is given by
\begin{eqnarray}\label{scalarcurv}
R=16\pi\mu.
\end{eqnarray}
From \eqref{principaleq}, \eqref{lapla} and \eqref{scalarcurv} a solution for the perfect fluid equation must satisfy 
\begin{eqnarray}\label{eqsemtraco}
f\mathring{R}ic=\mathring{\nabla}^{2}f,
\end{eqnarray}
where $\mathring{R}ic$ and $\mathring{\nabla}^{2}$ stand for the traceless Ricci and Hessian tensors, respectively. Furthermore, when we assume $\mu=\rho=0$ everywhere, we are referring to the static vacuum Einstein space-time. The set $\{f=0\}$ is well known to be the horizon, i.e., the event horizon of a static vacuum Einstein black hole. We further remark that $\{f=0\}$ may be defined as the set of limit points of Cauchy sequences on $(M^n,\,g)$ on which $f$ converges to $0$. Moreover, it should be emphasized that it is expected that $\{f=0\}\neq\emptyset$, since a complete
three-dimensional static vacuum space-time such that $f > 0$ everywhere must be trivial, i.e., the warped function $f$ must be constant and the space-time must be the flat Minkowsky space–time.

Astronomical evidence also indicates that the universe can be modeled (in smoothed, averaged form) as a space-time containing a perfect fluid whose “molecules” are the galaxies. At present, the dominant contribution to the density of the galactic fluid is the mass of the galaxies, with a much smaller pressure due mostly to radiation (see \cite[p. 341]{oneill1983}). Considering such background, in \cite{leandro2019} the authors proved that $\mu=\rho=0$ in the set $\{f=0\}$, provided that the dominant energy condition holds (i.e., $\mu\geq\rho\geq0$). Furthermore, $\{f=0\}$ is a minimal hypersurface for the static perfect fluid space-time (cf. \cite{coutinho2020,leandro2019}). Assuming that $\{f=0\}\neq\emptyset$ we will prove that an asymptotically flat static perfect fluid space-time must be spherically symmetric, if a particular equation of state is satisfied.

In this work we will consider similar asymptotic conditions used by Agostiniani and Mazzieri, Beig and Simon, K\"unzle, Masood-ul-Alam and Robinson (cf. \cite{agostiniani2017,beig1991,beig1992,kunzle1971,massod21987,robinson2009}, respectively), which are defined as follows.

\begin{definition}\label{def2}
	A solution $(M^{3},\,g,\,f,\,\mu,\,\rho)$ for \eqref{principaleq} and \eqref{lapla} is said to be asymptotically flat with one end $E$ if $M$ minus a compact set $K$ is diffeomorphic to $\mathbb{R}^{3}$ minus a closed ball, and the metric $g$, the lapse function $f$, the energy $\mu$ and the pressure $\rho$ satisfy the following asymptotic expansions at infinity.
	\begin{itemize}
		\item[(I)] Let $r^{2}=\displaystyle\sum_{i=1}^{3}x_{i}^{2}$, $x=(x_{1},\,x_{2},\,x_{3})\in M$, $\delta$ be the flat metric in which $\eta_{ij}(x)=o(r^{-2})$ and $\partial_{l}\eta_{ij}=o(r^{-3})$, as $r\rightarrow\infty$,  $$g_{ij}(x)=\delta_{ij}(x)+\eta_{ij}(x),$$
		where $1\leq l,\,i,\,j\leq n$.
		
		\item[(II)] For $\omega=o(r^{-2})$, as $r\rightarrow\infty$,
		\begin{eqnarray*}
			f=1-\dfrac{m}{r}+\omega(r).
		\end{eqnarray*}
		Moreover, $ \partial_{i}\omega=o(r^{-3})$, as $ r\rightarrow\infty$, and $\omega''(r)\leq \frac{2m}{r^3}$, where $1\leq  i \leq n$ and $m\in [0,\,\infty)$ represents the ADM mass.
	\item[(III)] 
	Consider $\mu\geq\rho\geq0$ and $\mu-\rho=o(r^{-4}),$ as $r\rightarrow\infty.$
\end{itemize}
We assume the existence of an interior boundary $\Sigma$ (non empty), where $f=0$ at $\Sigma$. Consider $\Sigma$ compact and such that $g$ and $f$ extend smoothly to $\Sigma$ (cf. Condition 2 in \cite{leandro2019}). 
\end{definition}

\begin{remark}\label{remark}
\
\begin{enumerate}
    \item The Schwarzschild solution is an example of static space which satisfies Definition \ref{def2}. In fact, we are assuming the same asymptotic conditions (Definition \ref{def2}-(I)-(II)) used by Robinson to prove the static vacuum black hole uniqueness theorem \cite{robinson1977,robinson2009}. Moreover, the asymptotic conditions assumed in $(I)$ and $(II)$ match with the asymptotics assumed in \cite{agostiniani2017,beig1991,beig1992,kunzle1971,massod21987,massod1987,robinson2009}. However, in the above definition we provided a different asymptotic condition for the pressure $\rho$ and density $\mu$. The most common hypothesis for the density and pressure is that they are identically zero outside the fluid region.

      \item The condition $\omega''(r)\leq \frac{2m}{r^3}$ also agrees with the second order derivatives for the function $\omega$ assumed by Robinson \cite{robinson1977,robinson2009}. Furthermore, it becomes trivial if we assume the equation of state $\mu+3\rho=0$ (cf. equation \eqref{beethoven}). However, we will not discuss this equation of state here since in this scenario the dominant energy condition does not hold (i.e., we may have negative pressure), and therefore we can not apply Theorem 1 in \cite{leandro2019} which was very important in the proof of our main result (Theorem \ref{princtheorem}).
\end{enumerate}
\end{remark}

 Here, we will discuss the following problem (cf. \cite{kunzle1971,massod21987,massod1987,massod2007}):
\begin{conjecture}\label{aiai}
	An asymptotically flat solution static perfect fluid space-time satisfying \eqref{principaleq} and \eqref{lapla} must be spherically symmetric.
\end{conjecture}

The main result (Theorem \ref{princtheorem}) will prove Conjecture \ref{aiai} considering that an equation of state holds. To do that, we first provide an divergence formula for the static perfect fluid equations and then, by an integration of such formula (cf. Lemma \ref{lema222200111}), we prove that the only possible static perfect solution must be trivial (i.e., Schwarzschild or $\mathbb{R}^{3}$). In fact, the conditions on the equation of state are not really relaxed here but the conditions on the asymptotic behavior (Definition \ref{def2}-(III)). To accomplish our goals, we assume that the isoperimetric (Penrose) inequality holds for $\Sigma$ (cf. \cite{huisken2001} and the discussion after equation \eqref{caseI}) .

Without further ado, we state our main result.

\begin{theorem}\label{princtheorem}
	An asymptotically flat solution for \eqref{principaleq} and \eqref{lapla} in which the energy-density is a smooth function of $f$
	and
	\begin{eqnarray}\label{SES}
(5\rho-\mu)fG+4F\frac{d\mu}{df}\geq0
	\end{eqnarray}
	must be isometric, in the exterior of a compact subset of $M$, either to Schwarzschild space or $\mathbb{R}^3$ with the standard Euclidean metric.
Here,
		\begin{eqnarray}\label{FG}
		F(f)=(cf^{2}+d)(1-f^{2})^{-3}\quad\mbox{and}\quad G(f)=6\left(\dfrac{F}{1-f^{2}}\right)-2c(1-f^{2})^{-3},
	\end{eqnarray}
	in which $c$ and $d$ are constants such that $F>0$.
	\end{theorem}
	
	It is worth to saying that functions $F$ and $G$ in the above theorem came from the original divergence formula of Robinson \cite[page 697]{robinson1977} used to proof the three dimensional static black hole uniqueness theorem. Here, we used these functions to get a distinct equation of state in the attempt to solve Conjecture \ref{aiai}.

As far was we know, the most general proof for Conjecture \ref{aiai} was presented by Masood-ul-Alam in \cite{massod2007}. In his proof, was assumed that the density $\mu(\rho)$ is a non-decreasing function of the pressure $\rho$, where $\mu$ and $\rho$ are functions of $f$ (see also \cite{massod1987,massod1988} and their assumptions over the equations of state). Is worth to saying that his technique avoids a Robinson-type identity and uses the positive mass theorem as an approach. On the other hand, the results of Beig and Simon \cite{beig1991,beig1992} used equations of state similar to \eqref{SES}. Assuming their equation of state holds, they were able to provide a rigidity result. In fact, the equation of state Beig and Simon studied is related to the divergence formula of Robinson. However, their intention was to use the maximum principle for the Laplacian of a conformal metric, which came from a divergence formula for the conformal metric. Then, they concluded that this conformal metric is, in fact, conformally flat. Those proofs follow the same trend of using the conformal metric to get a rigidity using the positive mass theorem. Our approach is more similar to the strategy used by Robinson to get the static vacuum black hole uniqueness theorem \cite{robinson1977,robinson2009}.

Now let us analyze the hypothesis assumed in our main theorem. First, bear in mind that equations of state like \eqref{SES} were considered before (cf. \cite{beig1991,beig1992,massod1988}). 
The decay (Definition \ref{def2}) assumed for $\mu-\rho$ came naturally since it is expected that $\mu=\rho=0$ outside the fluid region (cf. \cite{massod21987,massod1987,massod2007}). It seems more natural to assume a decay for the density and pressure instead of its immediate vanishing outside the fluid region. Thus, the assumption in Definition \ref{def2}-(III) is weaker than the one assumed in the earliest works about this topic.

In addition, the asymptotic condition for the second derivative of the function $\omega$ is reasonable, see Definition \ref{def2}-(II), if we compare this decay with the decay assumed by Robinson and Agostiniani-Mazzieri \cite{agostiniani2017,robinson1977,robinson2009}.
Hence, part of our hypothesis concerns a slight change of the asymptotic conditions considered by \cite{agostiniani2017,beig1991,beig1992,kunzle1971,massod21987} and others. It is also important to remember that the dominant energy condition $\mu\geq\rho$ holds for all known forms of matter.

\section{Background}

In this section we shall present some preliminaries which will be useful for the establishment of the desired results, we will obtain a useful expression of the divergence for the static perfect fluid equations similar to the divergence formula of Robinson \cite{robinson1977,robinson2009} that will be integrated later. Then, we will get an inequality which leads us to a contradiction with our divergence formula (Lemma \ref{lema222200111}). This will drive the solutions for the static perfect fluid space to be trivial (i.e., Schwarzschild or $\mathbb{R}^{3})$.

We start by recalling that for a Riemannian manifold $(M^{3},\,g)$ the curvature tensor is defined by the following decomposition formula
\begin{eqnarray*}
\label{weyl}
R_{ijkl}=\big(R_{ik}g_{jl}+R_{jl}g_{ik}-R_{il}g_{jk}-R_{jk}g_{il}\big)-\frac{R}{2}\big(g_{jl}g_{ik}-g_{il}g_{jk}\big),
\end{eqnarray*}
where $R_{ijkl}$ stands for the Riemannian curvature operator. Moreover, the Cotton tensor $C$ is given according  to

\begin{equation*}
\label{cotton}
\displaystyle{C_{ijk}=\nabla_{i}R_{jk}-\nabla_{j}R_{ik}-\frac{1}{4}\big(\nabla_{i}R g_{jk}-\nabla_{j}R g_{ik}).}
\end{equation*}

Another useful formula is the Ricci equation:
\begin{eqnarray}\label{riccieq}
\nabla_{i}\nabla_{j}\nabla_{k}f-\nabla_{j}\nabla_{i}\nabla_{k}f=R_{ijkl}\nabla^{l}f.
\end{eqnarray}
In what follows, we set the covariant 3-tensor $T_{ijk}$ by

\begin{eqnarray*}\label{TensorT}
T_{ijk}&=&2(\mathring{R}_{ik}\nabla_jf-\mathring{R}_{jk}\nabla_if)+(\mathring{R}_{jl}\nabla^{l}fg_{ik}-\mathring{R}_{il}\nabla^{l}fg_{jk}).
\end{eqnarray*}  
Notice, also, that it is skew-symmetric in the first two indices and trace-free in any two indices.

 The first lemma we present here was recently proved in \cite{coutinho2020}, and it connects the Cotton tensor with the perfect fluid equations in a simple expression which will be useful providing our divergence formula.

\begin{lemma}\label{Lem33}
Let $(M^3,g)$ be a Riemannian manifold and $f$ a smooth function on $M^3$ satisfying $f\mathring{R}ic=\mathring{\nabla^2}f.$ Then, it holds
\begin{eqnarray*}
fC_{ijk}=T_{ijk}.
\end{eqnarray*}
\end{lemma}

\
The next lemma gives us a formula for the norm of the Cotton tensor involving only the functions of Definition \ref{defA}.
\begin{lemma}\label{lem21}
	Let $\big(M^3,\,g,\,f)$ be the spatial factor of a static perfect fluid space-time. Then:
	\begin{eqnarray*}
	f^{4}|C|^{2}&=&4|\nabla f|^{2}[\Delta|\nabla f|^{2}-\frac{1}{f}\langle\nabla|\nabla f|^{2},\,\nabla f\rangle-8\pi f\langle\nabla\mu,\,\nabla f\rangle+8\pi(\mu+\rho)|\nabla f|^{2}-(\Delta f)^{2}]\nonumber\\
	&-&3|\nabla|\nabla f|^{2}|^{2}+4\Delta f\langle\nabla|\nabla f|^{2},\,\nabla f\rangle.
	\end{eqnarray*}
\end{lemma}
\begin{proof}
From Lemma \ref{Lem33} we have
		\begin{eqnarray*}
		f^{2}C_{ijk}&=&2(f\mathring{R}_{ik}\nabla_jf-f\mathring{R}_{jk}\nabla_if)+(f\mathring{R}_{jl}\nabla^{l}fg_{ik}-f\mathring{R}_{il}\nabla^{l}fg_{jk}).
	\end{eqnarray*}
Using \eqref{eqsemtraco} the above identity, it can be written in the following way
		\begin{eqnarray*}
	f^{2}C_{ijk}&=&2(\mathring{\nabla}^{2}_{ik}f\nabla_{j}f-\mathring{\nabla}^{2}_{jk}f\nabla_{i}f)+(\mathring{\nabla}^{2}_{jl}f\nabla^{l}fg_{ik}-\mathring{\nabla}^{2}_{il}f\nabla^{l}fg_{jk})\nonumber\\
\end{eqnarray*}

Now, to prove the next identity we only need to use the above equation. Then,
	\begin{eqnarray*}
	f^{4}|C|^{2}&=&f^{2}|T|^{2}\nonumber\\
	&=&8|\mathring{\nabla}^{2}f|^{2}|\nabla f|^{2}-12\mathring{\nabla}^{2}_{ik}f\nabla^{i}f(\mathring{\nabla}^{2}_{j})^{k}f\nabla^{j}f,
	\end{eqnarray*}
where $(\mathring{\nabla}^{2}_{j})^{k}f=\nabla_{j}\nabla^{k}f-\frac{\Delta f}{n}g_{j}^{k}.$
	
Since $\mathring{\nabla}^{2}f=\nabla^{2}f-\dfrac{\Delta f}{3}g$, $|\mathring{\nabla}^{2}f|^{2}=|\nabla^{2}f|^{2}-\dfrac{(\Delta f)^{2}}{3}$ and $\nabla^{2}f(\nabla f)=\frac{1}{2}\nabla|\nabla f|^{2}$ we get
	\begin{eqnarray}\label{normC}
	f^{4}|C|^{2}=8|{\nabla}^{2}f|^{2}|\nabla f|^{2}-3|\nabla|\nabla f|^{2}|^{2}-4|\nabla f|^{2}(\Delta f)^{2}+4\Delta f\langle\nabla|\nabla f|^{2},\,\nabla f\rangle.
	\end{eqnarray}

On the other hand, contracting  \eqref{riccieq} over $i$ and $k$ we get 
\begin{eqnarray*}
\nabla^{i}\nabla_{j}\nabla_{i}f-\nabla_{j}\Delta f=R_{jl}\nabla^{l}f.
\end{eqnarray*}
Moreover, from \eqref{principaleq} and \eqref{lapla} we have
\begin{eqnarray*}
\nabla^{i}\nabla_{j}\nabla_{i}f-\nabla_{j}\Delta f=\dfrac{1}{2f}\nabla_{j}|\nabla f|^{2}+4\pi(\mu-\rho)\nabla_{j}f.
\end{eqnarray*}
Then, using Equation (4.7) in \cite{coutinho2020},
\begin{eqnarray*}
\nabla_{j}\Delta f= \frac{f}{4}\nabla_{j}R-\frac{R}{2}\nabla_{j}f,
\end{eqnarray*}
we can infer that
\begin{eqnarray*}
\nabla^{i}\nabla_{j}\nabla_{i}f-\frac{f}{4}\nabla_{j}R+\frac{R}{2}\nabla_{j}f=\dfrac{1}{2f}\nabla_{j}|\nabla f|^{2}+4\pi(\mu-\rho)\nabla_{j}f.
\end{eqnarray*}
Thus,
\begin{eqnarray*}
\nabla^{j}f\nabla^{i}\nabla_{j}\nabla_{i}f=4\pi{f}\langle\nabla\mu,\,\nabla f\rangle+\dfrac{1}{2f}\langle\nabla|\nabla f|^{2},\,\nabla f\rangle-4\pi(\mu+\rho)|\nabla f|^{2}.
\end{eqnarray*}

So, using this in the next equation
\begin{eqnarray*}
\nabla^{i}(\nabla^{j}f\nabla_{j}\nabla_{i}f)=\nabla^{i}\nabla^{j}f\nabla_{i}\nabla_{j}f+\nabla^{j}f\nabla^{i}\nabla_{j}\nabla_{i}f,
\end{eqnarray*}
we will get
\begin{eqnarray*}\label{dt1}
2|\nabla^{2}f|^{2}=\Delta|\nabla f|^{2}-\frac{1}{f}\langle\nabla|\nabla f|^{2},\,\nabla f\rangle-8\pi f\langle\nabla\mu,\,\nabla f\rangle+8\pi(\mu+\rho)|\nabla f|^{2}.
\end{eqnarray*}
Combining the above equation with \eqref{normC} the result follows.

\end{proof}

\

From now on, we start developing the divergence formulas for the static perfect fluid. For this formula, which is the key ingredient to prove our main results, we need to define the functions $F$ and $G$ given by \eqref{FG}. This two functions appear naturally in the original proof of the divergence formula of Robinson \cite{robinson1977}. Here, we assume the existence of the same functions in the attempt to provide a good divergence formula for the perfect fluid equations.
\begin{lemma}\label{lema22}
	Let $\big(M^3,\,g,\,f)$ be the spatial factor of a static perfect fluid space-time such that $\mu:=\mu(\rho)$. Then,
	\begin{eqnarray*}
&&div\left[F(f^{-1}\nabla|\nabla f|^{2}+8\pi(\mu-\rho)\nabla f)+G|\nabla f|^{2}\nabla f\right]-\dfrac{Ff^{3}|C|^{2}}{4|\nabla f|^{2}}\nonumber\\
&=&16\pi F\langle\nabla\mu,\,\nabla f\rangle +\left[8\pi(\mu-\rho)|\nabla f|^{-2}F'+G'\right]|\nabla f|^{4}+\left[\dfrac{F\Delta f}{f}+8\pi(\mu-\rho)F+G|\nabla f|^{2}\right]\Delta f\nonumber\\
&+&\left(\frac{F'}{f}+G-\frac{F\Delta f}{f|\nabla f|^{2}}\right)\langle\nabla|\nabla f|^{2},\,\nabla f\rangle+\dfrac{3F|\nabla|\nabla f|^{2}|^{2}}{4f|\nabla f|^{2}},
	\end{eqnarray*}
	where $F(f)$ and $G(f)$ are smooth functions of $f$.
\end{lemma}
\begin{proof}
The Bianchi identity for $g$ is reduced to (cf. \cite{beig1992,massod1987,massod2007})
\begin{eqnarray}\label{masood}
\nabla\rho=-f^{-1}(\mu+\rho)\nabla f.
\end{eqnarray}
A straightforward computation gives us
	\begin{eqnarray*}
&&div\left[F(f^{-1}\nabla|\nabla f|^{2}+8\pi(\mu-\rho)\nabla f)+G|\nabla f|^{2}\nabla f\right]=f^{-1}\langle\nabla F,\,\nabla|\nabla f|^{2}\rangle+8\pi(\mu-\rho)\langle\nabla F,\,\nabla f\rangle\nonumber\\
&-&Ff^{-2}\langle\nabla f,\,\nabla|\nabla f|^{2}\rangle+Ff^{-1}\Delta|\nabla f|^{2}+8\pi F\langle\nabla\mu,\,\nabla f\rangle-8\pi F\langle\nabla\rho,\,\nabla f\rangle+8\pi(\mu-\rho)F\Delta f\nonumber\\
&+&|\nabla f|^{2}\langle\nabla G,\,\nabla f\rangle+G\langle\nabla f,\,\nabla|\nabla f|^{2}\rangle+G|\nabla f|^{2}\Delta f\nonumber\\
&=&f^{-1}F'\langle\nabla f,\,\nabla|\nabla f|^{2}\rangle+8\pi(\mu-\rho)F'|\nabla f|^{2}\nonumber\\
&-&Ff^{-2}\langle\nabla f,\,\nabla|\nabla f|^{2}\rangle+Ff^{-1}\Delta|\nabla f|^{2}+8\pi F\langle\nabla\mu,\,\nabla f\rangle-8\pi F\langle\nabla\rho,\,\nabla f\rangle+8\pi(\mu-\rho)F\Delta f\nonumber\\
&+&G'|\nabla f|^{4}+G\langle\nabla f,\,\nabla|\nabla f|^{2}\rangle+G|\nabla f|^{2}\Delta f\nonumber\\
&=&Ff^{-1}\Delta|\nabla f|^{2}+(f^{-1}F'-Ff^{-2}+G)\langle\nabla f,\,\nabla|\nabla f|^{2}\rangle+8\pi(\mu-\rho)F'|\nabla f|^{2}\nonumber\\
&+&8\pi F\langle\nabla\mu,\,\nabla f\rangle-8\pi F\langle\nabla\rho,\,\nabla f\rangle\nonumber\\
&+&G'|\nabla f|^{4}+[G|\nabla f|^{2}+8\pi(\mu-\rho)F]\Delta f\nonumber\\
&=&Ff^{-1}\Delta|\nabla f|^{2}+(f^{-1}F'-Ff^{-2}+G)\langle\nabla f,\,\nabla|\nabla f|^{2}\rangle\nonumber\\
&+&8\pi F\langle\nabla\mu,\,\nabla f\rangle+8\pi[ Ff^{-1}(\mu+\rho)+(\mu-\rho)F']|\nabla f|^{2}\nonumber\\
&+&[G|\nabla f|^{2}+8\pi(\mu-\rho)F]\Delta f+G'|\nabla f|^{4}.\nonumber\\
	\end{eqnarray*}
Combining the above formula with Lemma \ref{lem21} we get the result.
\end{proof}

Finally, we present the divergence equation for the static perfect fluid. This equality was mainly inspired by Robinson \cite{robinson1977}.
\begin{lemma}\label{lema222200111}
	Let $\big(M^3,\,g,\,f)$ be the spatial factor of a static perfect fluid space-time such that $\mu:=\mu(\rho)$. Then, 
	\begin{eqnarray*}\label{hiurieq}
&&div\left[F(f^{-1}\nabla|\nabla f|^{2}+8\pi(\mu-\rho)\nabla f)+G|\nabla f|^{2}\nabla f\right]\nonumber\\
&=&\dfrac{Ff^{3}|C|^{2}}{4|\nabla f|^{2}}-\frac{F\Delta f}{f|\nabla f|^{2}}\langle\nabla|\nabla f|^{2},\,\nabla f\rangle
+\dfrac{3F}{4f|\nabla f|^{2}}\left|\nabla|\nabla f|^{2}+8\dfrac{f|\nabla f|^{2}\nabla f}{(1-f^{2})}\right|^{2}\nonumber\\
&+&4\pi\left(3\mu+\rho\right)F\Delta f+\dfrac{96\pi{fF}}{1-f^{2}}|\nabla f|^{2}(\mu-\rho)+4\pi|\nabla f|^{2}\left[(5\rho-\mu)fG+4F\dfrac{d\mu}{df}\right].
	\end{eqnarray*}
\end{lemma}
\begin{proof}
Notice that 
	\begin{eqnarray}\label{normf}
	\dfrac{3F}{4f|\nabla f|^{2}}\left|\nabla|\nabla f|^{2}+8\dfrac{f|\nabla f|^{2}\nabla f}{(1-f^{2})}\right|^{2}&=&\dfrac{3F}{4f|\nabla f|^{2}}|\nabla|\nabla f|^{2}|^{2}+\dfrac{12F}{(1-f^{2})}\langle\nabla|\nabla f|^{2},\,\nabla f\rangle\nonumber\\
	&+&48\dfrac{Ff|\nabla f|^{4}}{(1-f^{2})^{2}}.
\end{eqnarray}
Combining Lemma \ref{lema22} with \eqref{normf}, it yields us
	\begin{eqnarray*}
&&div\left[F(f^{-1}\nabla|\nabla f|^{2}+8\pi(\mu-\rho)\nabla f)+G|\nabla f|^{2}\nabla f\right]-\dfrac{Ff^{3}|C|^{2}}{4|\nabla f|^{2}}\nonumber\\
&=&16\pi F\langle\nabla\mu,\,\nabla f\rangle +\left[8\pi(\mu-\rho)|\nabla f|^{-2}F'+G'-\frac{48Ff}{(1-f^{2})^{2}}\right]|\nabla f|^{4}+\left[\dfrac{F\Delta f}{f}+8\pi(\mu-\rho)F+G|\nabla f|^{2}\right]\Delta f\nonumber\\
&+&\left(\frac{F'}{f}+G-\frac{F\Delta f}{f|\nabla f|^{2}}-\dfrac{12F}{(1-f^{2})}\right)\langle\nabla|\nabla f|^{2},\,\nabla f\rangle+	\dfrac{3F}{4f|\nabla f|^{2}}\left|\nabla|\nabla f|^{2}+8\dfrac{f|\nabla f|^{2}\nabla f}{(1-f^{2})}\right|^{2}.
	\end{eqnarray*}
	
	Now, consider 
	\begin{eqnarray*}
		F(f)=(cf^{2}+d)(1-f^{2})^{-3},\quad G(f)=6\left(\dfrac{F}{1-f^{2}}\right)-2c(1-f^{2})^{-3}
	\end{eqnarray*}
	such that $c,\,d\in\mathbb{R}$.
A straightforward computation assures us that
$$\left\{
\begin{array}{ccc}
\frac{F'}{f}+G=\dfrac{12F}{1-f^{2}}; \\\\
G'=\dfrac{48Ff}{(1-f^{2})^{2}}.\\
\end{array}
\right.
$$

	Therefore, 
\begin{eqnarray*}\label{hiuri2}
&&div\left[F(f^
{-1}\nabla|\nabla f|^{2}+8\pi(\mu-\rho)\nabla f)+G|\nabla f|^{2}\nabla f\right]-\dfrac{Ff^{3}|C|^{2}}{4|\nabla f|^{2}}\nonumber\\
&=&16\pi F\langle\nabla\mu,\,\nabla f\rangle +8\pi(\mu-\rho)F'|\nabla f|^{2}-\frac{F\Delta f}{f|\nabla f|^{2}}\langle\nabla|\nabla f|^{2},\,\nabla f\rangle\nonumber\\
&+&	\dfrac{3F}{4f|\nabla f|^{2}}\left|\nabla|\nabla f|^{2}+8\dfrac{f|\nabla f|^{2}\nabla f}{(1-f^{2})}\right|^{2}+\left[4\pi(3\mu+\rho)F+G|\nabla f|^{2}\right]\Delta f,
	\end{eqnarray*}
where $F'=\frac{2cf}{(1-f^{2})^{3}}+\frac{6f(cf^{2}+d)}{(1-f^{2})^{4}}$. Then, we can rearrange the above equation to get
\begin{eqnarray*}
&&div\left[F(f^{-1}\nabla|\nabla f|^{2}+8\pi(\mu-\rho)\nabla f)+G|\nabla f|^{2}\nabla f\right]-\dfrac{Ff^{3}|C|^{2}}{4|\nabla f|^{2}}\nonumber\\
&=&16\pi F\langle\nabla\mu,\,\nabla f\rangle +[8\pi(\mu-\rho)F'+4\pi(\mu+3\rho)fG]|\nabla f|^{2}-\frac{F\Delta f}{f|\nabla f|^{2}}\langle\nabla|\nabla f|^{2},\,\nabla f\rangle\nonumber\\
&+&	\dfrac{3F}{4f|\nabla f|^{2}}\left|\nabla|\nabla f|^{2}+8\dfrac{f|\nabla f|^{2}\nabla f}{(1-f^{2})}\right|^{2}+4\pi\left(3\mu+\rho\right)F\Delta f\nonumber\\
&=&16\pi F\langle\nabla\mu,\,\nabla f\rangle +\dfrac{96\pi{fF}}{1-f^{2}}|\nabla f|^{2}(\mu-\rho)-\frac{F\Delta f}{f|\nabla f|^{2}}\langle\nabla|\nabla f|^{2},\,\nabla f\rangle\nonumber\\
&+&	\dfrac{3F}{4f|\nabla f|^{2}}\left|\nabla|\nabla f|^{2}+8\dfrac{f|\nabla f|^{2}\nabla f}{(1-f^{2})}\right|^{2}+4\pi\left(3\mu+\rho\right)F\Delta f+4\pi(5\rho-\mu)fG|\nabla f|^{2}.
	\end{eqnarray*}
Considering $\mu=\mu(f)$ and \eqref{masood}, we obtain the desired result.
\end{proof}

\section{Proof of the main results}

 \noindent {\bf Proof of Theorem \ref{princtheorem}:}
 We start the demonstration by proving that from \eqref{principaleq} we have
\begin{eqnarray*}\label{mozart}
X(|\nabla f|^2)&=&2\langle\nabla_X\nabla f,\nabla f\rangle\nonumber\\&=&2\nabla^2f(X,\nabla f)\nonumber\\&=&2fRic(X,\nabla f)-8\pi(\mu-\rho)f\langle X,\nabla f\rangle=0,
\end{eqnarray*} for any $X\in\mathfrak{X}(\Sigma).$ Since $f=0$ at $\Sigma$, $\kappa=|\nabla f|$ is a non null constant on $\Sigma$ (by lemma 1 in \cite{leandro2019} we have that $\nabla f$ does not vanish at $\Sigma$). 

In what follows, $\eta$ and $N=\dfrac{-\nabla f}{|\nabla f|}$ are the normal vector fields of the sphere $\mathbb{S}$ and interior boundary $\Sigma$, respectively. Here we assume $f$ and $g$ extend smoothly to the interior boundary $\Sigma$ (see Definition \ref{def2}). From Lemma \ref{lema222200111} and \eqref{lapla} we can infer that
\begin{eqnarray}\label{fern1}
\int_{M}div\left[F(f^{-1}\nabla|\nabla f|^{2}+8\pi(\mu-\rho)\nabla f)+G|\nabla f|^{2}\nabla f\right]dv\geq-4\pi\int_{M}\frac{(\mu+3\rho)F}{|\nabla f|^{2}}\langle\nabla|\nabla f|^{2},\,\nabla f\rangle dv.
\end{eqnarray}
From now on we apply Stokes's theorem in the above inequality but first notice that from \eqref{principaleq} we have
\begin{eqnarray*}
F(f^{-1}\nabla|\nabla f|^{2}+8\pi(\mu-\rho)\nabla f)+G|\nabla f|^{2}\nabla f=2FRic(\nabla f)+G|\nabla f|^{2}\nabla f.
\end{eqnarray*}
Therefore, from the fact that $f=0$ at $\Sigma$ and from the asymptotic conditions at the end $E$ of the manifold $M^3$, an integration of the above equation yields to

\begin{eqnarray*}
&&\int_{M}div\left[F(f^{-1}\nabla|\nabla f|^{2}+8\pi(\mu-\rho)\nabla f)+G|\nabla f|^{2}\nabla f\right]dv\nonumber\\
&&=\int_{\Sigma}\langle 2dRic(\nabla f)+2\left(3d-c\right)|\nabla f|^{2}\nabla f,\,\dfrac{-\nabla f}{|\nabla f|}\rangle ds\nonumber\\
&&+\lim_{r\rightarrow\infty}\int_{\mathbb{S}(r)}\langle Ff^{-1}\nabla|\nabla f|^{2}+8\pi(\mu-\rho)F\nabla f+G|\nabla f|^{2}\nabla f,\,\eta\rangle ds\nonumber\\
&&=-d|\nabla f|\int_{\Sigma}2 Ric(N,\,N)ds-2\left(3d-c\right)|\nabla f|^{3}Area(\Sigma)\nonumber\\
&&+\lim_{r\rightarrow\infty}\int_{\mathbb{S}(r)}\langle Ff^{-1}\nabla|\nabla f|^{2}+8\pi(\mu-\rho)F\nabla f+G|\nabla f|^{2}\nabla f,\,\eta\rangle ds.\nonumber\\
\end{eqnarray*}
A similar integration can be found in \cite[equation (3.24)]{massod21987} and \cite[equation (3.3)]{robinson1977}.

Since $\Sigma$ is umbilic (cf. \cite{coutinho2020}), i.e., the second fundamental formula $h^{\Sigma}=0$, and $R=16\pi\mu=0$ in $\Sigma$ (cf. Theorem 1 in \cite{leandro2019}, which proved $\mu=\rho=0$ at $\Sigma$ if $\mu$ and $\rho$ are non negative functions) from the Gauss equation we have
\begin{eqnarray*}
\label{EscalarBordoM}
R^{\Sigma}&=&-2Ric(N,\,N).
\end{eqnarray*}
Thus,
\begin{eqnarray}\label{f1}
&&\int_{M}div\left[F(f^{-1}\nabla|\nabla f|^{2}+8\pi(\mu-\rho)\nabla f)+G|\nabla f|^{2}\nabla f\right]dv\nonumber\\
&&=d\kappa\int_{\Sigma} R^{\Sigma}ds-2\left(3d-c\right)\kappa^{3}Area(\Sigma)\nonumber\\
&&+\lim_{r\rightarrow\infty}\int_{\mathbb{S}(r)}\langle Ff^{-1}\nabla|\nabla f|^{2}+8\pi(\mu-\rho)F\nabla f+G|\nabla f|^{2}\nabla f,\,\eta\rangle ds.
\end{eqnarray}

The easiest way to evaluate the two-dimensional integral at infinity that arises after the application of Stokes' theorem is to use spherical polar coordinates to describe the asymptotic flatness of the three-metric. Then, to evaluate asymptotically on a sphere of radius $r$ as $r$ tends to infinity - we will only need to keep leading terms. Using the facts that $g_{rr}$ tends to one and $f$ tends to $1-m/r$, we will compute the integral ahead.

 First of all, from Definition \ref{defA} and Definition \ref{def2} we have
\begin{eqnarray*}
	&&0\leq4\pi\int_{M}(\mu+3\rho)f dv=\int_{M}\Delta f dv=\int_{\Sigma} \langle\nabla f,\,\frac{-\nabla f}{|\nabla f|}\rangle ds+\displaystyle\lim_{r\rightarrow\infty}\int_{\mathbb{S}(r)}\langle\nabla f,\,\eta\rangle ds\nonumber\\
	&&=-\kappa Area(\Sigma)+m\displaystyle\lim_{r\rightarrow\infty}\dfrac{1}{r^{2}}\int_{\mathbb{S}(r)}ds=-\kappa Area(\Sigma)+4\pi{m},
\end{eqnarray*}
where $\eta$ is the exterior normal vector field of $\mathbb{S}$. So, we can conclude that
\begin{eqnarray}\label{areaasymp}
\kappa Area(\Sigma)\leq4\pi{m}.
\end{eqnarray}
Since $f>0$ in $M$, equality holds if and only if $\mu+3\rho=0$.

On the other hand, we can consult \cite{massod21987,robinson1977} to see that
\begin{eqnarray}\label{f2}
    \lim_{r\rightarrow\infty}\int_{\mathbb{S}(r)}\langle Ff^{-1}\nabla|\nabla f|^{2}+8\pi(\mu-\rho)F\nabla f+G|\nabla f|^{2}\nabla f,\,\eta\rangle ds \nonumber\\
    = \dfrac{-(c+d)\pi}{2m}+    \lim_{r\rightarrow\infty}\int_{\mathbb{S}(r)}\langle8\pi(\mu-\rho)F\nabla f,\,\eta\rangle ds\nonumber\\
    =\dfrac{-(c+d)\pi}{2m}+   \lim_{r\rightarrow\infty}\int_{\mathbb{S}(r)}8\pi(\mu-\rho)F\frac{m}{r^{2}} ds\nonumber\\
    =\dfrac{-(c+d)\pi}{2m}+ 32m\pi^{2}\lim_{r\rightarrow\infty}(\mu-\rho)F.
\end{eqnarray}
Therefore, combining \eqref{fern1}, \eqref{f1} and \eqref{f2} we get
\begin{eqnarray}\label{massa}
	&&d\kappa\int_{\Sigma}R^{\Sigma}ds-2\left(3d-c\right)\kappa^{3}Area(\Sigma)\nonumber\\
	&&\geq\dfrac{(c+d)\pi}{2m} - \left[32m\pi^{2}\lim_{r\rightarrow\infty}(\mu-\rho)F+4\pi\int_{M}\frac{(\mu+3\rho)F}{|\nabla f|^{2}}\langle\nabla|\nabla f|^{2},\,\nabla f\rangle dv\right].
\end{eqnarray}
Considering the asymptotic conditions we also have
\begin{eqnarray*}
\nabla f = f'\nabla r\quad\mbox{and}\quad \nabla|\nabla f|^{2}=2f'f''\nabla r;\quad r=|x|\rightarrow\infty.
\end{eqnarray*}
Notice that
\begin{eqnarray*}\label{dali}
\langle\nabla|\nabla f|^{2},\,\nabla f\rangle=2(f')^{2}f''.
\end{eqnarray*}

Then, by Definition \ref{def2}, we have $ \omega''\leq 2mr^{-3}$. So,
	 \begin{eqnarray}\label{beethoven}
	\frac{F\Delta f}{f|\nabla f|^{2}}\langle\nabla|\nabla f|^{2},\,\nabla f\rangle=8\pi{f}''(\mu+3\rho)F=\underbrace{8\pi(\mu+3\rho)F}_{\geq0}\underbrace{\left(\omega''-\frac{2m}{r^{3}}\right)}_{\leq0}\leq0.
 \end{eqnarray}
Additionally, assuming that $\mu-\rho=o(r^{-4})$, for a sufficiently large $r$, we get 
 \begin{eqnarray}\label{limittenso}
 \lim_{r\rightarrow\infty}(\mu-\rho)F=\lim_{r\rightarrow\infty}\dfrac{(\mu-\rho)[cf^{2}+d]}{(1-f^{2})^{3}}=\lim_{r\rightarrow\infty}\dfrac{r^{4}(\mu-\rho)[c(r-m)^{2}+dr^{2}]}{m^{3}(2r-m)^{3}}=0.
 \end{eqnarray}
Now, we need to consider two special cases: (I) $c=1$ and $d=0$;\,(II) $d=1$ and $c=-1$.

\

 \noindent {\bf Case (I):} Considering $c=1$ and $d=0$, from \eqref{massa} we have
 \begin{eqnarray}\label{caseI}
	\kappa^{3}Area(\Sigma)\geq\dfrac{\pi}{4m}. 
\end{eqnarray}
 Considering  \eqref{areaasymp}  and combining with the above inequality we get $$\frac{1}{4m}\leq\kappa.$$

 Therefore, using again \eqref{areaasymp} we have that the isoperimetric (Penrose) inequality holds, i.e., $Area(\Sigma)\leq16m^{2}\pi$. This result can be interpreted as an optimal lower bound for the mass $m$. In fact, from \eqref{areaasymp} we have $Area(\Sigma)\leq \frac{4m\pi}{\kappa}\leq16m^{2}\pi$. So, we can infer that $\kappa=|\nabla f|\Big|_{\Sigma}=\dfrac{1}{4m}$ (otherwise we will have a better estimate for $Area(\Sigma)$ than the Penrose estimate), and again from \eqref{caseI} we get
$$\kappa^{2}Area(\Sigma)\geq\pi.$$

  \noindent {\bf Case (II):} Considering $c=-1$ and $d=1$, from \eqref{massa} we have
\begin{eqnarray*}
	&&\kappa\int_{\Sigma}R^{\Sigma}ds-8\kappa^{3}Area(\Sigma)\geq 0
\end{eqnarray*} 

So, from the Gauss-Bonnet theorem we obtain
\begin{eqnarray*}
	2\pi\mathfrak{X}(\Sigma)=\int_{\Sigma}R^{\Sigma}ds\geq8\kappa^{2}Area(\Sigma)>0,
\end{eqnarray*} 
where $\mathfrak{X}(\Sigma)$ is the Euler characteristic of $\Sigma$. Thus, we can conclude that $\mathfrak{X}(\Sigma)$ is equal to $1$ or $2$. That is, $$\pi\geq\kappa^{2}Area(\Sigma).$$
  
  \noindent {\bf Conclusion:}
 Hence, {\bf Case (I)} and {\bf Case (II)} are compatible if and only if the right-hand side of the equality in Lemma \ref{lema222200111} is identically zero (cf. \cite{robinson1977,robinson2009}). That is, 
\begin{eqnarray}\label{bom}
0&=&\dfrac{Ff^{3}|C|^{2}}{4|\nabla f|^{2}}-\frac{F\Delta f}{f|\nabla f|^{2}}\langle\nabla|\nabla f|^{2},\,\nabla f\rangle
+\dfrac{3F}{4f|\nabla f|^{2}}\left|\nabla|\nabla f|^{2}+8\dfrac{f|\nabla f|^{2}\nabla f}{(1-f^{2})}\right|^{2}\nonumber\\
&+&4\pi\left(3\mu+\rho\right)F\Delta f+\dfrac{96\pi{fF}}{1-f^{2}}|\nabla f|^{2}(\mu-\rho)+4\pi|\nabla f|^{2}\left[(5\rho-\mu)fG+4F\dfrac{d\mu}{df}\right]\nonumber\\
&=&\dfrac{Ff^{3}|C|^{2}}{4|\nabla f|^{2}}
+\dfrac{3F}{4f|\nabla f|^{2}}\left|\nabla|\nabla f|^{2}+8\dfrac{f|\nabla f|^{2}\nabla f}{(1-f^{2})}\right|^{2}+8\pi(\mu+3\rho)F\left(\frac{2m}{r^{3}}-\omega''\right)\nonumber\\
&+&4\pi\left(3\mu+\rho\right)F\Delta f+\dfrac{96\pi{fF}}{1-f^{2}}|\nabla f|^{2}(\mu-\rho)+4\pi|\nabla f|^{2}\left[(5\rho-\mu)fG+4F\dfrac{d\mu}{df}\right]\geq0.
\end{eqnarray}

Therefore, from the above identity we have two possibilities, either the Cotton tensor is zero and $\mu=\rho=0$, or the Cotton tensor is zero and $f$ is  constant. 

In the first case we get that $(M^{3},\,g,\,f)$ is conformally flat and the static space is vacuum. 
Then, from \cite{bunting1987,robinson1977} we have that $(M^3,\,g)$ is isometric to Schwarzschild. If $f$ is a constant function, we have that $(M^3,\,g)$ is conformally flat and, from \eqref{eqsemtraco}, an Einstein manifold, so it has constant curvature (cf. \cite{kobayashi1980}). However, from \eqref{lapla} we get $\mu+3\rho=0$, and since we assume that $\mu$ and $\rho$ are non negative, we must have $\mu=\rho=0$. Thus, the only possibility is $(M^3,\,g)$ to be isometric to $\mathbb{R}^3$ with the Euclidean metric.
\hfill$\Box$

\

\

\

 \noindent {\bf Acknowledgment:}
 {\it
This work was done while the third author was a postdoc at Instituto de Matemática e Estatística, Universidade Federal de Goiás, Brazil. He is grateful to the hosted institution for the scientific atmosphere that it has provided during his visit. 
}

\

\end{document}